\documentclass[12pt, oneside]{amsart}

\usepackage{amsfonts}
\usepackage{latexsym}
\usepackage{amsmath}
\usepackage{amsthm}
\usepackage{amssymb}
\usepackage[usenames,dvipsnames]{xcolor}
\usepackage{hyperref}
\usepackage{stmaryrd}
\usepackage{caption}
\usepackage{tikz-cd}
\usepackage{ulem}
\usepackage[all]{xy}
\usetikzlibrary{matrix,arrows,backgrounds,shapes.misc,shapes.geometric,fit}

\hypersetup{
    bookmarks=true,         
    unicode=false,          
    pdftoolbar=true,        
    pdfmenubar=true,        
    pdffitwindow=false,     
    pdfstartview={FitH},    
    pdftitle={My title},    
    pdfauthor={Author},     
    pdfsubject={Subject},   
    pdfcreator={Creator},   
    pdfproducer={Producer}, 
    pdfkeywords={keyword1} {key2} {key3}, 
    pdfnewwindow=true,      
    colorlinks=true,       
    linkcolor=Sepia,        
    citecolor=Red,        
    filecolor=magenta,      
    urlcolor=cyan           
}

\newtheorem{thm}{Theorem}[section]
\newtheorem{prop}[thm]{Proposition}
\newtheorem{lemma}[thm]{Lemma}
\newtheorem{cor}[thm]{Corollary}

\newtheorem{remark}[thm]{Remark}
\newtheorem{defn}[thm]{Definition}

\numberwithin{equation}{section}

\makeatletter
\def\imod#1{\allowbreak\mkern5mu{\operator@font mod}\,\,#1}
\makeatother

\def\makeop#1{\expandafter\def\csname#1\endcsname
  {\mathop{\rm #1}\nolimits}\ignorespaces}
\makeop{Hom}   \makeop{End}   \makeop{Aut}   \makeop{Isom}  \makeop{Pic}
\makeop{Gal}   \makeop{ord}   \makeop{Char}  \makeop{Div}   \makeop{Lie}
\makeop{PGL}   \makeop{Corr}  \makeop{PSL}   \makeop{sgn}   \makeop{Spf}
\makeop{Spec}  \makeop{Map}    \makeop{Nm}    \makeop{Fr}    \makeop{disc}
\makeop{Proj}  \makeop{supp}  \makeop{ker}   \makeop{im}    \makeop{dom}
\makeop{coker} \makeop{Stab}  \makeop{SO}    \makeop{SL}    \makeop{SL}
\makeop{Cl}    \makeop{cond}  \makeop{Br}    \makeop{inv}   \makeop{rank}
\makeop{id}    \makeop{Fil}   \makeop{Frac}  \makeop{GL}    \makeop{SU}
\makeop{Trd}   \makeop{Sp}    \makeop{Tr}    \makeop{Trd}   \makeop{diag}
\makeop{Res}   \makeop{ind}   \makeop{depth} \makeop{Tr}    \makeop{st}
\makeop{Ad}    \makeop{Int}   \makeop{tr}    \makeop{Sym}   \makeop{can}
\makeop{length}\makeop{SO}    \makeop{torsion} \makeop{GSp} \makeop{Ker}
\makeop{Adm}   \makeop{Frob}  \makeop{id}    \makeop{Tor}   \makeop{Ind}
\makeop{CoInd} \makeop{Inf}   \makeop{Cor}   \makeop{CoInf} \makeop{coLie}
\makeop{Ann}

\allowdisplaybreaks

\newcommand{\Addresses}{{
  \bigskip
  \footnotesize

\textsc{Department of Mathematics, Johns Hopkins University}\par\nopagebreak
  \textit{E-mail address} : \texttt{xwang151@math.jhu.edu}

}}

\begin{document}
\author{Xiyuan Wang}

\title{Weight elimination in two dimensions when $p=2$}

\date{}


\maketitle

\begin{abstract}
We prove the `weight elimination' part of the weight part of Serre's conjecture for mod 2 Galois representations for rank two unitary groups, by modifying the results in \cite{GLS14} and \cite{GLS15}.
\end{abstract}

\section{Introduction}
Let $p$ be a prime number and $F$ be a number field with ring of integers $\mathcal{O}_F$. We let $G_{F}$ denote the absolute Galois group of $F$. Assume that $\overline{r}:G_{F}\rightarrow \text{GL}_2(\overline{\mathbb{F}}_p)$ is an irreducible and modular Galois representation.
When $F$ is unramified at $p$, a conjecture of Buzzard-Diamond-Jarvis (see \cite{BDJ10}), proved by \cite{GLS14, GK14, New13, BLGG13} when $p>2$, predicts the set of irreducible representations $a$ of $\text{GL}_2(\mathcal{O}_{F}/p )$ such that $\overline{r}$ is modular of weight $a$. We recall the shape of this conjecture. For each place $v\mid p$, let $F_{v}$ be the completetion of $F$ at $v$ and denote the residue field of $F_{v}$ by $k_{v}$. For each representation $\overline{\rho}:G_{F_{v}}\rightarrow \text{GL}_{2}(\overline{\mathbb{F}}_p)$, Buzzard-Diamond-Jarvis define a set $W^{\mathrm{explicit}}(\overline{\rho})$ of irreducible representations of $\text{GL}_2(k_v)$ (called Serre  weights), and conjecture that $\overline{r}$ is modular of weight $a$ if and only if $a\in W^{\mathrm{explicit}}(\overline{r}):=\otimes_{v\mid p} W^{\mathrm{explicit}}(\overline{r}|_{G_{F_{v}}})$

The papers \cite{GLS14, GLS15} in particular prove the `weight elimination' direction of this conjecture (as well as its natural generalization to the case where $F$ is ramified at $p$), i.e., that if $\overline{r}$ is modular of weight $a$ (in the sense of \cite[Def. 2.1.9]{BLGG13}, i.e., arising from a suitable automorphic form on a rank two unitary group), then $a\in W^{\mathrm{explicit}}(\overline{r})$. They achieve this goal by proving a purely local result, namely that for another set of Serre weights $W^{\mathrm{cris}}(\overline{\rho})$ (defined in terms of the Hodge-Tate weights of crystalline lifts of the local representation $\overline{\rho}$), one has $W^{\mathrm{cris}}(\overline{\rho})\subseteq W^{\mathrm{explicit}}(\overline{\rho})$. 

The main reason they require $p>2$ is that they use a field extension lemma (see \cite[Lem. 5.1.2]{Liu08}) which relies on this condition. In this note, we remove the condition $p>2$ by proving Lemma 2.1 below. This note is closely related to the papers \cite{GLS14} and \cite{GLS15}. We follow exactly the same notations and ideas in these two papers. One should treat this note as a complement of \cite{GLS14, GLS15} when $p=2$. The main result of this note is as follows.
\begin{thm}[Theorem 6.2]
Let $F$ be an imaginary CM field with maximal totally real subfield $F^{+}$. Assume that $F/F^{+}$ is unramified at all finite places, that every place of $F^{+}$ above $2$ splits in $F$, and that $[F^{+}:\mathbb{Q}]$ is even. Suppose that $\overline{r}:G_{F}\rightarrow \GL_{2}(\overline{\mathbb{F}}_{2})$ is an irreducible representation with split ramification that is modular in the sense of \cite[Def. 2.1.9]{BLGG13}. 

Let $a$ be a Serre weight. If $\overline{r}$ is modular of weight $a$, then $a\in W^{\mathrm{explicit}}(\overline{r})$.  
\end{thm}
We now explain the structure of this note. In Section 2, we prove the main lemma of this note. In Section 3, we recall some sets of Serre weights. They are important for the formulation and the proof of the weight part of Serre's conjecture. In Section 3 and 4, we revise \cite{GLS14, GLS15} to prove the main local result of this paper. One should read these two sections together with \cite{GLS14} and \cite{GLS15}. In the last section, we prove our main global result.

I would like to thank Hui Gao for helpful discussions. The results in this note were proved in 2015 as part of the author's Ph.D. project. Now I decide to publish this note since the results in this note has already been used in several papers (see \cite{DS17}, \cite{CEGM17}, \cite{EG19}, \cite{Bar20}, and \cite{Ste20}).

The converse to weight elimination, namely that $\overline{r}$ is modular of each predicted weight, appears to be rather more difficult, because of limitations on $2$-adic automorphy lifting theorems. However, we will report on some results in this direction in a future paper.
\subsection{Notations}
Let $p$ be a prime number. Let $K$ be a finite extension of $\mathbb{Q}_p$, with ring of integers $\mathcal{O}_K$ and residue field $k$. Let $K_0$ be the maximal unramified extension of $\mathbb{Q}_p$ contained in $K$. Let $\pi$ be a uniformizer of $K$. Let $E(u)$ denote the minimal polynomial of $\pi$ over $K_0$ , and set $e=\deg E(u)$. Let $W(k)$ be the ring of Witt vectors which is the ring of integers in $K_0$. We write $I_K$ for the inertia subgroup of the absolute Galois group $G_{K}$ of $K$. 

If $W$ is a de Rham representation of $G_{K}$ over $\overline{\mathbb{Q}}_p$ and $\kappa:K\hookrightarrow \overline{\mathbb{Q}}_p$ is an embedding, then by definition the multiset $\text{HT}_{\kappa}(W)$ of Hodge-Tate weights of $W$ with respect to $\kappa$ contains $i$ with multiplicity $\dim_{\overline{\mathbb{Q}}_p}(W\otimes_{\kappa , K}\hat{\overline{K}}(-i))^{G_K}$ with the usual notations for Tate twists. Thus for example $\text{HT}_{\kappa}(\varepsilon)=\{1\}$, where $\varepsilon$ is the $p$-adic cyclotomic character. We will refer to the elements of $\mathrm{HT}_{\kappa}(W)$  as the ``$\kappa$-labeled Hodge Tate weights''.

Define $\mathfrak{G}=W(k)[[u]]$. The ring $\mathfrak{G}$ is equipped with a Frobenius endomorphism $\varphi$ via $u\mapsto u^p$ along with the usual (arithmetic) Frobenius on $W(k)$. Let $S$ be the $p$-adic completion of the divided power envelope of $W(k)[u]$ with respect to the ideal generated by $E(u)$. Write $S_{K_0}=S[1/p]$.

Let $R=\varprojlim \mathcal{O}_{\overline{K}}/p$ where the transition maps are the $p$th power maps. $W(R)$ is the ring of Witt vectors of $R$. We have a unique surjective projection map $\theta: W(R)\rightarrow \hat{\mathcal{O}}_{\overline{K}}$ to the $p$-adic completion of $\mathcal{O}_{\overline{K}}$. We denote by $B_{\rm{dR}}^{+}$ the $\ker(\theta)$-completion of $W(R)[1/p]$. We define filtration on $W(R)$ by 
\[
\mathrm{Fil}^{i}W(R)=W(R)\cap (\ker(\theta))^{i}B_{\rm{dR}}^{+}.
\]
Let $A_{\text{cris}}$ be the $p$-adic completion of the divided power envelope of $W(R)$ with respect to $\ker(\theta)$. We write $B_{\mathrm{cris}}^{+}=A_{\mathrm{cris}}[1/p]$.

We fix a compatible system of $p^n$th roots of $\pi$. Set $\uline{\pi}=(\pi_n)_{n\geqslant 0}\in R$, $\uline{\epsilon}=(\zeta_{p^n})_{n\geqslant 0}\in R$, and $t=-\text{log}([\uline{\epsilon}])\in A_{\text{cris}} $. For any $g\in G_K$, write $\uline{\epsilon}(g)=g(\uline{\pi})/\uline{\pi}$. By \cite[Ex. 5.3.3]{Liu07} there exists an element $\mathfrak{t}\in W(R)$ such that $t=c\varphi (\mathfrak{t})$ with $c\in S^{\times}$. Following \cite[\S 5]{Fon94} we define
\[
I^{[m]}W(R)=\{x\in W(R): \varphi^{n}(x)\in \text{Fil}^{m}B_{\mathrm{cris}}^{+} \text{ for all } n>0\}.
\]
(See \cite[\S 5]{Fon94} for the definition of the filtration on $B_{\mathrm{cris}}^{+}$.) By \cite[Prop. 5.1.3]{Fon94} the ideal $I^{[m]}W(R)$ is principal, generated by $\varphi (\mathfrak{t})^{m}$.

We define 
\[
K_{\infty}=\bigcup_{n=0}^{\infty}K(\pi_n),\ K_{p^{\infty}}=\bigcup_{n=1}^{\infty}K(\zeta_{p^n}),\ \hat{K}=\bigcup_{n=1}^{\infty}K_{\infty}(\zeta_{p^n}).
\]
Write $G_{\infty}=\text{Gal}(\overline{K}/K_{\infty})$, $\hat{G}_{p^{\infty}}=\text{Gal}(\hat{K}/K_{p^{\infty}})$, $\hat{G}=\text{Gal}(\hat{K}/K)$, and $H_K=\text{Gal}(\hat{K}/K_{\infty})$.

\section{The main lemma}
In this section, we let $p=2$. We have the following main lemma of this note.
\begin{lemma}
There is a uniformizer $\varpi$ of $K$, such that $K_{2^{\infty}}\cap K_{\infty}=K$, where $K_{\infty}=\bigcup_{n=1}^{\infty}K(\sqrt[2^n]{\varpi})$. 
\end{lemma}
\begin{proof}
By \cite[Prop. 4.1.5]{Liu10}, for any uniformizer $\pi$ of $K$, we have $K_{2^{\infty}}\cap K_{\infty}=K$ or $K(\sqrt{\pi})$. So we just need to show that we can choose a uniformizer $\varpi$, such that $K(\sqrt{\varpi})\not\subseteq K_{2^{\infty}}$.

Assume that $[K:\mathbb{Q}_2]=d$. By Kummer theory, there are $2^{d+1}$ quadratic extensions of $K$ which are of the form $K(\sqrt{\pi})/K$, where $\pi$ is a uniformizer of $K$.

On the other hand, $\text{Gal}(K_{2^{\infty}}/K)$ is a subgroup of $\mathbb{Z}_{2}^{\times}\cong \mathbb{Z}/2\times \mathbb{Z}_2$. One can see that there are at most three quadratic extensions of $K$ inside $K_{2^{\infty}}$. Since $2^{d+1}\geq 4> 3$, we can find a uniformizer $\varpi$, such that $K(\sqrt{\varpi})\not\subseteq K_{2^{\infty}}$. We are done.
\end{proof}
Note that, when $p>2$, one has $K_{p^{\infty}}\cap K_{\infty}=K$ for any choice of uniformizer $\pi$.

As a corollary of the above lemma, we have $\hat{G}\simeq \hat{G}_{2^{\infty}}\rtimes H_{K}$.
To complete the proof of weight elimination when $p=2$, it suffices to make several modifications of some results in \cite{GLS14} and \cite{GLS15}.

\section{Serre weights}
In this section, we recall some sets of Serre weights.
Note that $p$ can be any prime number in this section. $K$ is a finite extension of $\mathbb{Q} _p$ with ring of integers $\mathcal{O}_{K}$ and residue field $k$. We set $f=[k: \mathbb{F}_p]$.

\begin{defn}
A Serre weight is an isomorphism class of irreducible $\overline{\mathbb{F}}_p$-representation of $\GL_{2}(k)$.
\end{defn}
We know that any such representation, up to isomorphism, is of the form
\[
F_a:=(\otimes_{\sigma: k\hookrightarrow \overline{\mathbb{F}}_p} \text{det}^{a_{\sigma ,2}}\otimes \Sym^{a_{\sigma , 1}- a_{\sigma , 2}}k^{2})\otimes_{\sigma , k}\overline{\mathbb{F}}_p ,
\]
where $0\leq a_{\sigma , 1}-a_{\sigma ,2}\leq p-1$ for each $\sigma$. We have $F_a\cong F_b$ as representations of $\GL_{2}(k)$ if and only if $a_{\sigma ,1}-a_{\sigma ,2}=b_{\sigma ,1}-b_{\sigma ,2}$ for each $\sigma$, and the character $k^{\times}\rightarrow \overline{\mathbb{F}}_{p}^{\times}$, defined by $x\mapsto \prod_{\sigma k:\hookrightarrow \overline{\mathbb{F}}_p}\sigma (x)^{a_{\sigma ,2}-b_{\sigma ,2}}$, is trivial.

We write $\mathbb{Z}_{+}^{2}$ for the set of pairs of integers $(a_1, a_2)$ with $0\leq a_1 -a_2\leq p-1$ and say $a=(a_{\sigma ,1}, a_{\sigma ,2})_{\sigma}\in (\mathbb{Z}_{+}^{2})^{\Hom(k,\overline{\mathbb{F}}_p)}$ is a Serre weight. We identify the Serre weight $a\in (\mathbb{Z}_{+}^{2})^{\Hom(k, \overline{\mathbb{F}}_p)}$ with the Serre weight represented by $F_{a}$. We say that $\lambda\in (\mathbb{Z}_{+}^{2})^{\Hom_{\mathbb{Q}_p}(K, \overline{\mathbb{Q}}_p)}$ is a lift of $a\in (\mathbb{Z}_{+}^{2})^{\Hom(k,\overline{\mathbb{F}}_p)}$ if for each $\sigma\in \Hom(k,\overline{\mathbb{F}}_p)$ there is an element $\kappa\in \Hom_{\mathbb{Q}_p}(K, \overline{\mathbb{Q}}_p)$ lifting $\sigma$ such that $\lambda_{\kappa}=a_{\sigma}$, and for other $\kappa^{'}\in \Hom_{\mathbb{Q}_p}(K, \overline{\mathbb{Q}}_p)$ lifting $\sigma$ we have $\lambda_{\kappa^{'}}=0$. 

\begin{defn}
Let $K/ \mathbb{Q}_p$ be a finite extension, let $\lambda\in (\mathbb{Z}_{+}^{2})^{\Hom_{\mathbb{Q}_p}(K, \overline{\mathbb{Q}}_p)}$, and let $\rho:G_K\rightarrow \GL_2 (\overline{\mathbb{Q}}_p)$ be a de Rham representation. Then we say that $\rho$ has Hodge type $\lambda$ if for each $\kappa\in \Hom_{\mathbb{Q}_p}(K, \overline{\mathbb{Q}}_p)$ we have $\mathrm{HT}_{\kappa}(\rho)=\{\lambda_{\kappa ,1}+1, \lambda_{\kappa ,2}\}$. 
\end{defn}
Now we want to define several sets of Serre weights (see Section 4 of \cite{BLGG13}). The point is that, conjectually, they are the same. And it should be the set of predicted Serre weights.

We firstly introduce the predicted sets of Serre weights $W^{\mathrm{cris}}(\overline{\rho})$.
\begin{defn}
Let $\overline{\rho}: G_{K}\rightarrow \GL_2(\overline{\mathbb{F}}_p)$ be a continuous representation. Then we let $W^{\mathrm{cris}}(\overline{\rho})$ be the set of Serre weights $a\in (\mathbb{Z}_{+}^{2})^{\Hom(k, \overline{\mathbb{F}}_p)}$ with the property that there is a crystalline representation $\rho:G_K\rightarrow \GL_2(\overline{\mathbb{Q}}_p)$ lifting $\overline{\rho}$, such that $\rho$ has Hodge type $\lambda$ for some lift $\lambda\in (\mathbb{Z}_{+}^{2})^{\Hom_{\mathbb{Q}_p}(K, \overline{\mathbb{Q}}_p)}$ of $a$.
\end{defn}

Now we want to give a lower bound for the set $W^{\mathrm{cris}}(\overline{\rho})$. Let  $K^{'}$ be the quadratic unramified extension of $K$ inside $\overline{\mathbb{Q}}_p$ with residue field $k^{'}$. 

\begin{defn}
If $\overline{\rho}:G_{K}\rightarrow \GL_2(\overline{\mathbb{F}}_p)$ is irreducible, then we let $W^{\mathrm{explicit}}(\overline{\rho})$ be the set of Serre weights $a\in (\mathbb{Z}_{+}^{2})^{\Hom(k, \overline{\mathbb{F}}_p)}$, such that there is a subset $J\subset \Hom(k^{'}, \overline{\mathbb{F}}_p)$ contaning exactly one element extending each element of $\Hom(k,\overline{\mathbb{F}}_p)$, for each $\sigma\in \Hom (k, \overline{\mathbb{F}}_p)$ and integer $0\leq \delta_{\sigma}\leq e-1$, and, if $\Hom(k^{'}, \overline{\mathbb{F}}_p)=J\coprod J^{c}$, then 
\[
\overline{\rho}|_{I_{K}}\cong
\left(
\begin{array}{cc}
\prod_{\sigma\in J}\omega_{\sigma}^{a_{\sigma , 1}+1+\delta_{\sigma}}\prod_{\sigma \in J^{c}}\omega_{\sigma}^{a_{\sigma ,2}+e-1-\delta_{\sigma}} & 0\\
0 & \prod_{\sigma\in J^{c}}\omega_{\sigma}^{a_{\sigma ,1}+1+\delta_{\sigma}}\prod_{\sigma\in J}\omega^{a_{\sigma ,2}+e-1-\delta_{\sigma}}\\
\end{array}
\right). 
\]    

If $\overline{\rho}:G_K\rightarrow \GL_2(\overline{\mathbb{F}}_p)$ is reducible, then we let $W^{\mathrm{explicit}}(\overline{\rho})$ be the set of Serre weights $a\in (\mathbb{Z}_{+}^{2})^{\Hom(k, \overline{\mathbb{F}}_p)}$, such that $\overline{\rho}$ has a crystalline lift of the form 
\[
\left(
\begin{array}{cc}
\psi_1 & \ast \\
0 & \psi_2
\end{array}
\right)
\] 
which has Hodge type $\lambda$ for some lift $\lambda\in (\mathbb{Z}_{+}^{2})^{\Hom_{\mathbb{Q}_p}(K, \overline{\mathbb{Q}}_p)}$ of $a$. In particular, if $a\in W^{\mathrm{explicit}}(\overline{\rho})$, then it is necessarily the case that there is a decomposition $\Hom(k,\overline{\mathbb{F}}_p)=J\coprod J^c$ and for each $\sigma\in \Hom(k,\overline{\mathbb{F}}_p)$ there is an integer $0\leq \delta_{\sigma}\leq e-1$ such that
\[
\overline{\rho}|_{I_K}\cong 
\left(
\begin{array}{cc}
\prod_{\sigma\in J}\omega_{\sigma}^{a_{\sigma ,1}+1+\delta_{\sigma}}\prod_{\sigma\in J^{c}}\omega_{\sigma}^{a_{\sigma ,2}+\delta_{\sigma}} & \ast \\
0 & \prod_{\sigma\in J^{c}}\omega_{\sigma}^{a_{\sigma ,1}+e-\delta_{\sigma}}\prod_{\sigma\in J}\omega_{\sigma}^{a_{\sigma ,2}+e-1-\delta_{\sigma}}
\end{array}
\right).
\]
\end{defn}
If $K$ is a finite unramified extension of $\mathbb{Q}_p$, then $W^{\mathrm{explicit}}(\overline{\rho})$ is exactly the set $W^{\mathrm{BDJ}}(\overline{\rho})$ defined in Section 2 of \cite{GLS14}. By the above definitions and \cite[Lem. 4.1.19]{BLGG13}, we have 
\[
W^{\mathrm{explicit}}(\overline{\rho})\subseteq W^{\mathrm{cris}}(\overline{\rho}).
\]

We have the following main local result of \cite{GLS14, GLS15}.
\begin{thm}
Let $\overline{\rho}: G_K\rightarrow \GL_2(\overline{\mathbb{F}}_p)$ be a continuous representation.
If $p>2$, then
$W^{\mathrm{cris}}(\overline{\rho})\subseteq W^{\mathrm{explicit}}(\overline{\rho})$.
\end{thm} 
In the following two sections, we remove the condition $p>2$.

\section{Revision of \texorpdfstring{\cite{GLS14}}{} when \texorpdfstring{$p=2$}{}}
In this section, We revise the paper \cite{GLS14} for the case $p=2$. From now on, we fix the uniformizer $\pi=\varpi$ in Lemma 2.1 and let $p=2$. With this choice of uniformizer, the conclusions of \cite[Lem. 5.1.2]{Liu08} still hold when $p=2$. We have $\hat{G}\simeq \hat{G}_{2^{\infty}}\rtimes H_{K}$. So we can choose (and fix) a topological generator $\tau\in \hat{G}_{2^{\infty}}$, such that $\uline{\epsilon}(\tau)=\tau (\uline{\pi})/\uline{\pi}=\uline{\epsilon}$. 

Let $T$ be a $G_K$-stable $\mathbb{Z}_2$-lattice in a semi-stable representation $V$ of dimension $d$ with Hodge-Tate weights in $[0,2]$. Let
$\mathfrak{M}$ be the Kisin module attached to $T$. Let $\mathcal{D}=S_{K_0}\otimes_{\varphi, \mathfrak{G}}\mathfrak{M}$ be the Breuil module attached to $\mathfrak{M}$ and $N$ be the monodromy operator on $\mathcal{D}$. We regard $\mathfrak{M}$ as an $\varphi(\mathfrak{G})$-submodule of $\mathcal{D}$. Select a $\varphi(\mathfrak{G})$-basis $\hat{e}_1,\ldots ,\hat{e}_d$ of $\mathfrak{M}$. We have $N(\hat{e}_1,\ldots ,\hat{e}_d)=(\hat{e}_1,\ldots,\hat{e}_d)U$ with $U$ a matrix with coefficients in $S_{K_0}$. Let $\tilde{S}=W(k)[[u^2,\frac{u^{2e}}{2}]]$

\begin{prop}
We have $U\in M_{d\times d} (\tilde{S}[\frac{1}{2}])$. If $V$ is crystalline, then furthermore $U\in u^2M_{d\times d}(\tilde{S}[\frac{1}{2}]\cap S)$. In particular, the hypothesis $p>2$ can be removed from \cite[Prop. 4.7 and Prop. 5.9]{GLS14}
\end{prop}
\begin{proof}
The statement that $U\in M_{d\times d}(\tilde{S}[\frac{1}{2}])$ has been proved in \cite[Prop. 4.7]{GLS14}. Now we assume that $V$ is crystalline. By the first part of the proof of Proposition 4.7 in \cite{GLS14}, we have $U\in u^2M_{d\times d}(\tilde{S}[\frac{1}{2}])$. Let $U=u^2U^{'}$. We need to show $U^{'}\in M_{d\times d}(S)$.

Following exactly the same idea of the proof of Proposition 2.13 in \cite{Liu12}, for any $x\in \mathcal{D}$ we have
\begin{equation}
\tau(x)=\sum_{i=0}^{\infty}\gamma_i(t)\otimes N^i(x).
\end{equation}
Since $\tau\in \hat{G}_{2^{\infty}}$ acts trivially on $\uline{\epsilon}$, it acts trivially on $t$. By computations we have
\begin{equation}
(\tau -1)^n(x)=\sum_{m=n}^{\infty}\left(\sum_{i_1+\cdots +i_n=m,i_j\geqslant 1}\frac{m!}{i_1!\cdots i_n!}\right)\gamma_m(t)\otimes N^m(x).
\end{equation}  
Following the proof of Proposition 4.7 in \cite{GLS14}, we have $(\tau -1)^n(x)\in u^2I^{[n]}W(R)\otimes_{\varphi ,\mathfrak{G}}\mathfrak{M}$.

 We next show that $((\tau -1)^n/ntu^2)(x)$ is well defined in $A_{\text{cris}}\otimes _{\varphi , \mathfrak{G}} \mathfrak{M}$ and $((\tau -1)^n/ntu^2)(x)\rightarrow 0$ $2$-adically as $n\rightarrow \infty$. Note that $t=c\varphi(\mathfrak{t})$ with $\varphi(\mathfrak{t})$ a generator of $I^{[1]}W(R)$. It suffices to show that for any $n$, $(\varphi(\mathfrak{t}))^{n-1}/n$ is in $A_{\text{cris}}$ and it goes to zero $2$-adically. Note that $\varphi (\mathfrak{t})\in \mathrm{Fil}^{2}W(R)+2W(R)$, we have $\varphi(\mathfrak{t})/2 \in A_{\text{cris}}$. So $\varphi(\mathfrak{t})^3/4=2(\varphi(\mathfrak{t})/2)^3\in A_{\text{cris}}$. This verifies that $(\varphi (\mathfrak{t}))^{n-1}/n \in A_{\mathrm{cris}}$ when $n=2$ or $4$. When $n$ is not $2$ or $4$, we let $n=2^sm$ with $2\nmid m$. We have
\begin{equation}
\frac{\varphi(\mathfrak{t})^{n-1}}{n}=\frac{\varphi(\mathfrak{t})^{n-1}}{(n-1)!}\frac{(2^sm-1)!}{2^sm}.
\end{equation}
Since $(\varphi(\mathfrak{t})^{i})/i!$ goes to zero $2$-adically, it suffices to show that $(2^sm-1)!/2^sm$ is in $\mathbb{Z}_2$ when $n$ is not $2$ or $4$. If $s=0$, this is trivial. If $s=1$, then $m\geqslant 3$. So $v_2((2m-1)!)\geqslant v_2(5!)\geqslant 1$. If $s=2$, then $m\geqslant 3$. So $v_2((4m-1)!)\geqslant v_2(11!)\geqslant 2$. If $s>2$, then $v_2((2^sm-1)!)\geqslant v_2((2^s-1)!)\geqslant v_2(2^{s-1}\cdot2)=s=v_2(2^sm)$. 

By a formal computation, we get 
\begin{equation}
\frac{N(x)}{u^2}=\sum_{n=1}^{\infty}(-1)^{n-1}\frac{(\tau -1)^{n}}{ntu^2}(x).
\end{equation}
So we have proved $N(x)/u^2 \in A_{\text{cris}}\otimes _{\varphi , \mathfrak{G}} \mathfrak{M}$. Since $A_{\text{cris}}\cap \tilde{S}[\frac{1}{2}]\subset A_{\text{cris}}\cap S[\frac{1}{2}]=S$, we are done. 
\end{proof}  

Now following the proof in \cite{GLS14}, we can show that the structure theorem for Kisin modules \cite[Thm. 4.22]{GLS14} still holds when $p=2$.

\begin{thm}
Assume that $K$ is unramified with uniformizer $\pi$ chosen as in Lemma 2.1, $V$ is crystalline with Hodge-Tate weights in $[0,2]$. Then there exists an $\mathcal{O}_{E}[[u]]$-basis $\{e_{j,s}\}$ of $\mathfrak{M}$ such that
\begin{enumerate}
\item[$\bullet$] $e_{1,s},\ldots ,e_{d,s}$ is an $\mathcal{O}_E[[u]]$-basis of $\mathfrak{M}_s$ for each $0\leq s \leq f-1$, and 
\item[$\bullet$] we have 
                      \[
                      \varphi(e_{1,s-1},\ldots , e_{d,s-1})=(e_{1,s}, \ldots ,e_{d,s})X_s \Lambda_s Y_s
                      \]
                      where $X_s$ and $Y_s$ are invertible matrices, $Y_s$ is congruent to the identity matrix modulo $2$, and $\Lambda_s$ is matrix $[E(u)^{r_{1,s}}, \ldots ,E(u)^{r_{d,s}}]$.
\end{enumerate}
\end{thm}
\begin{proof}
The hypothesis $p>2$ is only used in one place in the proof of \cite[Thm. 4.22]{GLS14}, namely in the proof of \cite[Prop. 4.7]{GLS14} (see the footnote in the proof of that item), and our Proposition 4.1 provides the missing argument in that case.
\end{proof}

Now that we have extended \cite[Thm. 4.22]{GLS14} to the case $p=2$, the proof of \cite[Thm. 7.9]{GLS14} carries over to the case $p=2$; one must only adjust the statement to account for the additional special case $(r_0,\ldots  ,r_{f-1})=(2, \ldots  ,2)$, $J=\{0, \ldots ,f-1\}$, and $a=b$ that one finds in \cite[Thm. 7.4]{GLS14}. To be precise, one obtains the following.

\begin{thm}
Suppose that $K/\mathbb{Q}_2$ is unramified with uniformizer $\pi$ chosen as in Lemma 2.1. Let $T$ be a $G_K$-stable $\mathcal{O}_E$-lattice in crystalline representation $V$ of $E$-dimension $2$ whose $\kappa_s$-labeled Hodge-tate weights are $\{0,r_s\}$ with $r_s\in [1,2]$ for all $s$. Let $\mathfrak{M}$ be the Kisin module attached to $T$, and let $\overline{\mathfrak{M}}:=\mathfrak{M}\otimes_{\mathcal{O}_E}k_E$.

Assume that the $k_E[G_K]$-module $\overline{T}:=T/\mathfrak{m}_ET$ is reducible. Then $\overline{\mathfrak{M}}$ is an extension of two $\varphi$-modules of rank one, and there exist $a,b\in k^{\times}_E$ and a subset $J\subset \{0,\ldots , f-1\}$ so that $\overline{\mathfrak{M}}$ is as follows.

Set $h_i=r_i$ if $i\in J$ and $h_i=0$ if $i\not\in J$. Then $\overline{\mathfrak{M}}$ is an extension of $\overline{\mathfrak{M}}(h_0,\ldots , h_{f-1} ; a)$ by $\overline{\mathfrak{M}}(r_0-h_0,\ldots , r_{f-1}-h_{f-1} ; b)$, and we can choose bases $e_i, f_i$ of the $\overline{\mathfrak{M}}_i$ so that $\varphi$ has the form
\[
\varphi (e_{i-1})= (b)_iu^{r_i-h_i}e_i
\]
\[
\varphi (f_{i-1})= (a)_iu^{h_i}f_i +x_i e_i
\]
with $x_i=0$ if $i\not\in J$ and $x_i\in k_E$ constant if $i\in J$, except in the following two cases:
\begin{enumerate}
\item[$\bullet$] $(r_0, \ldots , r_{f-1})\in \mathcal{P}$, $J=\{i: r_{i-1}\not= 2\}$, and $a=b$, or
\item[$\bullet$] $(r_0, \ldots, r_{f-1})=(2, \ldots , 2)$, $J=\{0, \ldots ,f-1\}$, and $a=b$.
\end{enumerate}
In that case fix $i_0\in J$; then $x_i$ may be taken to be $0$ for all $i\not\in J$, to be a constant for all $i$ except $i=i_0$, and to be the sum of a constant and a term of degree $2$ (for the first exceptional case) or degree $4$ (for the second exceptional case).

Finally, $\overline{T}|_{I_K}\simeq
\left(
\begin{array}{cc}
\prod_{i\in J}\omega_{i}^{r_i} & \ast\\
0 & \prod_{i\not\in J}\omega_{i}^{r_i}\\
\end{array}
\right)$. 
\end{thm}
\begin{proof}
The proof of the second exceptional case is exactly the same as the first exceptional case. See the proof of \cite[Thm. 7.9]{GLS14}.
\end{proof}

\begin{cor}
Suppose that $K/\mathbb{Q}_2$ is unramfied. Let $\overline{\rho}:G_K\rightarrow \GL_2(\overline{\mathbb{F}}_2)$ be the reduction mod $2$ of a $G_K$-stable $\overline{\mathbb{Z}}_2$-lattice in a crystalline $\overline{\mathbb{Q}}_2$-representation of dimension 2 whose $\kappa$-labeled Hodge-Tate weights are $\{0,r_{\kappa}\}$ with $r_{\kappa}\in [1,2]$ for all $\kappa$.

Assume that $\overline{\rho}$ is reducible. Let $S=\Hom(k,\overline{\mathbb{F}}_2)$, and identify the set $S$ with $\Hom_{\mathbb{Q}_2}(K, \overline{\mathbb{Q}}_2)$. Then there is a subset $J\subset S$ such that
\[
\overline{\rho}|_{I_K}\simeq
\left(
\begin{array}{cc}
\prod_{i\in J}\omega_{i}^{r_i} & \ast\\
0 & \prod_{i\not\in J}\omega_{i}^{r_i}\\
\end{array}
\right)
\]
\end{cor}
\begin{proof}
This follows immediately from Theorem 4.3.
\end{proof}

The discussion of \cite[\S 8]{GLS14} now carries over directly to the case $p=2$, with the following minor modification. In the definition of $J_{\text{max}}$ preceding \cite[Prop. 8.8]{GLS14}, according to \cite[Lem. 7.1]{GLS14} there is one more case to consider when $p=2$, namely when $r_i=2$ for all $i$ and $J=\varnothing$ or $\{0,\ldots , f-1 \}$. In this case we set $J_{\text{max}}=\{0,\ldots , f-1\}$. Then we have

\begin{prop}
Let $\hat{\overline{\mathfrak{M}}}$ be a $(\varphi , \hat{G})$-module of type $(\vec{r}, a,b, J)$, and set $h=h(J)$. Then there exists a $(\varphi , \hat{G})$-module $\hat{\overline{\mathfrak{N}}}$ of type $(\vec{r},a,b,J_{\text{max}})$, such that $\hat{T}(\hat{\overline{\mathfrak{N}}}) \simeq \hat{T}(\hat{\overline{\mathfrak{M}}})$.
\end{prop} 
\begin{proof}
If $\vec{r}=(2,\ldots ,2)$ and $J=\varnothing$, then the ambient Kisin module $\overline{\mathfrak{M}}$ is split, and the extension $\hat{\overline{\mathfrak{M}}}$ is also split. So in this case there is nothing to prove. For the other cases, see the proof of \cite[Prop. 8.8]{GLS14}
\end{proof}
We have the following Theorem.
\begin{thm}
Let $K/\mathbb{Q}_2$ be a finite unramified extension and $\rho:G_K\rightarrow \GL_2(\overline{\mathbb{Z}}_2)$ be a continuous representation such that $\overline{\rho}:G_K\rightarrow \GL_2(\overline{\mathbb{F}}_2)$ is reducible. Suppose that $\rho$ is crystalline with $\kappa$-Hodge-Tate weights $\{b_{\kappa ,1}, b_{\kappa ,2}\}$ for each $\kappa\in \Hom_{\mathbb{Q}_2}(K, \overline{\mathbb{Q}}_2)$, and suppose further that $1\leq b_{\kappa ,1}-b_{\kappa , 2}\leq 2$ for each $\kappa$.

Then there is a reducible crystalline representation $\rho':G_K\rightarrow \GL_2(\overline{\mathbb{Z}}_2)$ with the same $\kappa$-Hodge-Tate weights as $\rho$ for each $\kappa$, such that $\overline{\rho}\simeq \overline{\rho}'$.
\end{thm}
\begin{proof}
Write $
\overline{\rho}\simeq
\left(
\begin{array}{cc}
\overline{\psi}_1 & \ast\\
0 & \overline{\psi}_2\\
\end{array}
\right)
$. By Corollary 4.4, there is a decomposition $\text{Hom}_{\mathbb{Q}_2}(K, \overline{\mathbb{Q}}_2)=J\coprod J^{c}$ such that $\overline{\psi}_1|_{I_K}=\prod_{\kappa\in J}\omega_{\overline{\kappa}}^{b_{\kappa , 1}}\prod_{\kappa\in J^{c}}\omega_{\overline{\kappa}}^{b_{\kappa , 2}}$ and $\overline{\psi}_2|_{I_K}=\prod_{\kappa\in J^c}\omega_{\overline{\kappa}}^{b_{\kappa , 1}}\prod_{\kappa\in J}\omega_{\overline{\kappa}}^{b_{\kappa , 2}}$. If $b_{\kappa, 1}-b_{\kappa, 2}=2$ for all $\kappa$ and $(\overline{\psi}_1\overline{\psi}_{2}^{-1})|_{I_K}=1$, the result follows from \cite[Lem. 9.4]{GLS14}, whose proof allows $p=2$ (note that the mod $2$ cyclotomic character is trivial). In all other cases, thanks to Theorem 4.3 and Proposition 4.5 we can proceed as in the proof of \cite[Thm. 9.1]{GLS14}.
\end{proof}

\begin{cor}
If $K/\mathbb{Q}_2$ is a finite unramified extension, then
\[
W^{\mathrm{cris}}(\overline{\rho})\subseteq W^{\mathrm{explicit}}(\overline{\rho}).
\] 
\end{cor}
\begin{proof}
If $\overline{\rho}$ is reducible, this is an immediate consequence of Theorem 4.6 and Definition 3.4. If $\overline{\rho}$ is irreducible, it can be deduced from the reducible case following the idea in Section 10 of \cite{GLS14}.   
\end{proof}

\section{Revision of \texorpdfstring{\cite{GLS15}}{} when \texorpdfstring{$p=2$}{}}

In this section, we revise the paper \cite{GLS15} for the case $p=2$. Note that $p$ is $2$ in this section. We use the same terminology and notations as in \cite{GLS15}; in particular we refer the reader to \cite[Def. 2.3.1]{GLS15} for the definition of a pseudo-BT representation. The goal is to prove our main local result that $W^{\mathrm{cris}}(\overline{\rho})\subseteq W^{\mathrm{explicit}}(\overline{\rho})$ in the case that $K/\mathbb{Q}_2$ is a possibly ramified finite extension.

We have our main structure theorem for Kisin modules associated to pseudo-BT representations, thus generalizing \cite[Thm. 2.4.1]{GLS15} to the case $p=2$.
\begin{thm}
Assume that $K$ is a finite extension of $\mathbb{Q}_p$ with uniformizer $\pi$ chosen as in Lemma 2.1. Let $\mathfrak{M}$ be the Kisin module corresponding to a lattice in a pseudo-BT representation $V$ of weight $\{r_i \}$. Then there exists an $\mathcal{O}_E[[u]]$-basis $\{\mathfrak{e}_i, \mathfrak{f}_i\}$ of $\mathfrak{M}_i$ for all $0\leq i \leq f-1$ such that
\[
\varphi (\mathfrak{e}_{i-1}, \mathfrak{f}_{i-1})=(\mathfrak{e}_i, \mathfrak{f}_i)X_i \left(
\prod_{i=1}^{e-1}\Lambda_{i,e-j}Z_{i,e-j}
\right)\Lambda_{i,0}Y_i,
\]
for $X_i$,$Y_i\in \text{GL}_2(\mathcal{O}_{E}[[u]])$ with $Y_i\equiv I_2 (\text{mod } \mathfrak{m}_E)$, matrices $Z_{i,j}\in GL_2(\mathcal{O}_E)$ for all $j$, and $
\Lambda_{i,0}=
\left(
\begin{array}{cc}
1 & 0\\
0 & (u-\pi_{i,0})^{r_i}\\
\end{array}
\right)$
and
$
\Lambda_{i,j}=
\left(
\begin{array}{cc}
1 & 0\\
0 & u-\pi_{ij}\\
\end{array}
\right)$ for $j=1,\ldots, e-1$.

\end{thm}
\begin{proof}
The discussion of \cite{GLS15} allows $p=2$ until Theorem 2.3.5(2). However, the only place the proof of \cite[Thm. 2.3.5(2)]{GLS15} uses $p\not=2$ is in an application of \cite[Cor. 4.11]{GLS14}, which now applies in our setting thanks to our work in the previous section. Thus Theorem 2.3.5(2) of \cite{GLS15} holds in our setting when $p=2$. It follows that \cite[Cor. 2.3.10]{GLS15} holds as well. Since these are the only inputs to the proof of \cite[Thm. 2.4.1]{GLS15} that required $p\neq 2$, the result follows.
\end{proof}

\begin{cor}
The Proposition 3.1.3 of \cite{GLS15} holds without the assumption that $p\neq 2$ (with suitably chosen uniformizer when $p=2$).
\end{cor}
\begin{proof}
This is an immediate consequence of Theorem 5.1.
\end{proof}

If $V$ is a pesudo-BT representation of weight $\{r_i\}$ and $\lambda=\overline{\kappa}_i \in \text{Hom}(k,\overline{\mathbb{F}}_2),$ we write $r_{\lambda}:=r_i$. Let $k_2$ denote the unique quadratic extension of $k$ inside the residue field of $\overline{K}$. If $\lambda\in\text{Hom}(k_2, \overline{\mathbb{F}}_2)$, we write $r_{\lambda}:=r_{\lambda |_k}$.

\begin{thm}
Let $T$ be a lattice in pseudo-BT representation V of weight $\{r_i\}$, and assume that $\overline{T}=T\otimes_{\mathcal{O}_E} k_E$.

If $\overline{T}$ is reducible, then there is a subset $J\subseteq \text{Hom}(k,\overline{\mathbb{F}}_p)$ and integers $x_\lambda\in [0,e-1]$ such that
\[
\overline{T}|_{I_K}\simeq
\left(
\begin{array}{cc}
\prod_{\lambda\in J}\omega_{\lambda}^{r_\lambda +x_\lambda} \prod_{\lambda\not\in J}\omega_{\lambda}^{e-1-x_\lambda}& \ast\\
0 & \prod_{\lambda\not\in J}\omega_{\lambda}^{r_\lambda +x_\lambda} \prod_{\lambda\in J}\omega_{\lambda}^{e-1-x_\lambda}\\
\end{array}
\right)
\]

If $\overline{T}$ is absolutely irreducible, then there is a balanced subset $J\subseteq \text{Hom}(k_2,\overline{\mathbb{F}}_p)$ and integers  $x_\lambda\in [0,e-1]$ such that $x_\lambda$ depends only on $\lambda |_k$ and
\[
\overline{T}|_{I_K}\simeq
\left(
\begin{array}{cc}
\prod_{\lambda\in J}\omega_{\lambda}^{r_\lambda +x_\lambda} \prod_{\lambda\not\in J}\omega_{\lambda}^{e-1-x_\lambda}& 0\\
0 & \prod_{\lambda\not\in J}\omega_{\lambda}^{r_\lambda +x_\lambda} \prod_{\lambda\in J}\omega_{\lambda}^{e-1-x_\lambda}\\
\end{array}
\right)
\]
\end{thm}
\begin{proof}
If $\overline{T}$ is reducible, see the proof of \cite[Thm. 3.1.4]{GLS15}, which now goes through in our setting thanks to Corollary 5.2. In the irreducible case, the case $e\geq 2(=p)$ is handled exactly as in the proof of \cite[Thm. 3.1.5]{GLS15}. The case $e=1$ is a consequence of the results in $\S 4$.
\end{proof}

Except in the application of \cite[Thm. 2.4.1]{GLS15} to the proof of \cite[Thm. 5.1.5]{GLS15}, none of the arguments in Section 5 of \cite{GLS15} make any use of the hypothesis $p\neq 2$. Thus all of the results in that section hold when $p=2$ (and the uniformizer of $K$ is suitably chosen). Similarly, since we have checked that \cite[Cor. 5.10]{GLS14} still holds when $p=2$, the same is true of \cite[Lem. 6.12, Lem. 6.1.3, Prop. 6.1.7]{GLS15}. Finally, we have our main local theorem.
\begin{thm}
Suppose that $\overline{\rho}: G_K\rightarrow \GL_2(\overline{\mathbb{F}}_2)$ is a continuous representation. Then $W^{\mathrm{explicit}}(\overline{\rho})= W^{\mathrm{cris}}(\overline{\rho})$.
\end{thm}
\begin{proof}
If $\overline{\rho}$ is semisimple, thanks to our Theorem 5.3 the result follows as in the proof of \cite[Thm. 4.1.6]{GLS15}. If $\overline{\rho}$ is not a twist of an extension of the trivial character by the cyclotomic character, the result follows as in the proof of \cite[Thm. 5.4.1]{GLS15}. If $\overline{\rho}$ is a twist of an extension of the trivial character by the cyclotomic character, the result follows as in the proof of \cite[Thm. 6.1.8]{GLS15}.
\end{proof}

We conclude this section by checking that the set of Serre weights $W^{\mathrm{explicit}}(\overline{\rho})$ depends only on the restriction to inertia $\overline{\rho}|_{I_K}$ when $p=2$.
\begin{prop}
Assume that $K/\mathbb{Q}_2$ is a finite extension. Let $\overline{\rho}, \overline{\rho}':G_{K}\rightarrow \GL_2(\overline{\mathbb{F}}_2)$ be two continuous representations with $\overline{\rho}|_{I_{K}}\cong \overline{\rho}'|_{I_{K}}$. Then $W^{\mathrm{explicit}}(\overline{\rho})=W^{\mathrm{explicit}}(\overline{\rho}')$.
\end{prop}  
\begin{proof}
Thanks to our results in Section 4 and 5, the proof of \cite[Prop. 6.3.1]{GLS15} carries over to the case $p=2$.
\end{proof}
\section{Global results}
In this section, we prove the main global result in this note. Let $F$ be an imaginary CM field with maximal totally real subfield $F^{+}$. For each place $w$ of $F$ above $p$, let $k_{w}$ denote the residue field of $F_{w}$. If $w$ lies over a place $v$ of $F^{+}$, we write $v=ww^{c}$. Write $S:=\coprod_{w|p}\Hom(k_w, \overline{\mathbb{F}}_p)$, and let $(\mathbb{Z}_{+}^{2})_{0}^{S}$ denote the subset of $(\mathbb{Z}_{+}^{2})^{S}$ consisting of elements $a$ such that for each $w|p$, if $\sigma\in \Hom(k_w, \overline{\mathbb{F}}_{p})$ then
\[
a_{\sigma ,1}+a_{\sigma c ,2}=0.
\]
If $a\in (\mathbb{Z}_{+}^{2})^{S}$ and $w|p$ is a place of $F$, then let $a_w$ denote the element $(a_{\sigma})_{\sigma\in \Hom(k_{w}, \overline{\mathbb{F}}_p)}$ of $(\mathbb{Z}_{+}^{2})^{\Hom(k_w , \overline{\mathbb{F}}_p)}$.

Let $\overline{r}:G_F\rightarrow \GL_2(\overline{\mathbb{F}}_p)$ be a continuous irreducible representation.
\begin{defn}
We define $W^{\mathrm{explicit}}(\overline{r})$ (resp. $W^{\mathrm{cris}}(\overline{r})$) to be the set of elements $a\in (\mathbb{Z}_{+}^{2})_{0}^{S}$ such that, for each $w$ above $p$, $a_w\in W^{\mathrm{explicit}}(\overline{r}|_{G_{F_{w}}})$ (resp. $W^{\mathrm{cris}}(\overline{r}|_{G_{F_{w}}})$).
\end{defn}
We now establish weight elimination in the weight part of Serre's conjecture for rank two unitary groups when $p=2$. In particular, we follow \cite[Def. 2.1.9]{BLGG13} for the definition of what it means for $\overline{r}$ to be modular, and more precisely for $\overline{r}$ to be modular of some Serre weight $a$. Now we can state the main result of this note.
\begin{thm}
Let $F$ be an imaginary CM field with maximal totally real subfield $F^{+}$. Assume that $F/F^{+}$ is unramified at all finite places, that every place of $F^{+}$ above $2$ splits in $F$, and that $[F^{+}:\mathbb{Q}]$ is even. Suppose that $\overline{r}:G_{F}\rightarrow \GL_{2}(\overline{\mathbb{F}}_{2})$ is an irreducible modular representation with split ramification that is modular in the sense of \cite[Def. 2.1.9]{BLGG13}. 

Let $a$ be a Serre weight. If $\overline{r}$ is modular of weight $a$, then $a\in W^{\mathrm{explicit}}(\overline{r})$.  
\end{thm} 
\begin{proof}
This theorem is a consequence of \cite[Thm. 4.1.3]{BLGG13} (cf. also \cite[Lem. 2.1.11]{BLGG13}) by our Theorem 5.4 and Definition 6.1. Note that in Theorem 4.1.3 of \cite{BLGG13}, the hypothesis $p>2$ is required beacuse they assume the same condition at the beginning of Section 2. But one can check that everything is right in the proof of \cite[Lem. 2.1.11]{BLGG13} when $p=2$ (see e.g. Section 4 of \cite{Tho17}).
\end{proof}

\begin{remark}
Let $W(\overline{r})$ be the set of Serre weights in which $\overline{r}$ is modular. The proof of the weight part of Serre's conjecture for rank two unitary groups in \cite{GLS14, GLS15, BLGG13} proceeds by proving the chain of inclusions
\[
W^{\mathrm{explicit}}(\overline{r})\subseteq W(\overline{r})\subseteq W^{\mathrm{cris}}(\overline{r})\subseteq W^{\mathrm{explicit}}(\overline{r}).
\]
This note verified this chain when $p=2$ except the first inclusion. When $p>2$, the first inclusion is one of the main results of \cite{BLGG13}. Their proof uses a `change of weight' automorphy lifting theorem for $p$-adic Galois representations developed in \cite{BLGGT14}. However, such a theorem is not presently available when $p=2$.  
\end{remark}

\bigskip

\vskip 0.5 cm


\vskip 0.5 cm



\Addresses


\end{document}